%% file: Boundary_rectifiability.tex
\documentclass[reqno,11pt]{amsart} 

\usepackage{color}
\usepackage{xcolor}

\usepackage{ifpdf}
\ifpdf 
\usepackage{graphicx,import}   % to include graphics
\usepackage[pdftex,     % sets up hyperref to use pdftex driver
plainpages=false,   % allows page i and 1 to exist in the same document
breaklinks=true,    % link texts can be broken at the end of line
colorlinks=true,
linkcolor=blue,
 citecolor=magenta,
pdftitle={Boundary rectifiability },
pdfauthor={Giacomo Del Nin}
]{hyperref} 
\else 
\usepackage{graphicx}       % to include graphics
\usepackage{hyperref}       % to simplify the use of \href
\fi 

\usepackage{amsfonts,amsmath}	
\usepackage{amssymb}
\usepackage{verbatim}
\usepackage{amsopn}
\usepackage[english]{babel}
\usepackage{amsthm}
\usepackage{enumerate}
\usepackage{mathrsfs}	
\usepackage{mathtools}
\usepackage{harpoon}
\usepackage{esint}
\usepackage{todonotes}

\usepackage{stmaryrd}
\usepackage{bm}
\usepackage{bbm}
\usepackage{caption}
\usepackage{subcaption}
\usepackage{nicefrac}
\usepackage[a4paper,left=35mm,right=35mm,top=40mm,bottom=40mm,marginpar=25mm]{geometry}

\newcommand{\Dcal}{\mathcal{D}}

\newcommand{\Lcal}{\mathcal{L}}

\newcommand{\Tcal}{\mathcal{T}}

\newcommand{\Msf}{\mathsf{M}}

\newcommand{\altnorm}[1]{{\left\vert\kern-0.25ex\left\vert\kern-0.25ex\left\vert #1 \right\vert\kern-0.25ex\right\vert\kern-0.25ex\right\vert}}
\newcommand{\eps}{\varepsilon}

\newcommand{\ibf}{\mathbf{i}}

\newcommand{\Tan}{\mathrm{Tan}}

\theoremstyle{definition} \newtheorem{definition}{Definition}[section]
\theoremstyle{definition} \newtheorem{remark}[definition]{Remark}
\theoremstyle{plain} \newtheorem{lemma}[definition]{Lemma}
\theoremstyle{plain} \newtheorem{proposition}[definition]{Proposition}
\theoremstyle{plain} \newtheorem{theorem}[definition]{Theorem}
\theoremstyle{plain} \newtheorem{corollary}[definition]{Corollary}
\theoremstyle{definition} 
\theoremstyle{plain} 
\theoremstyle{definition} 
\theoremstyle{plain} 
\theoremstyle{plain}

\DeclareMathOperator{\dist}{dist}
\DeclareMathOperator{\supp}{supp}

\DeclareMathOperator{\Wedge}{{\textstyle\bigwedge}}

\newcommand{\R}{\mathbb{R}}

\newcommand{\N}{\mathbb{N}}
\newcommand{\Z}{\mathbb{Z}}

\newcommand{\loc}{\text{\rm loc}}
\newcommand{\D}{\mathcal{D}}

\newcommand{\M}{\mathsf{M}}

\newcommand{\1}{\mathbbm 1}

\newcounter{counter}

\newcommand{\curr}[1]{[\![#1]\!]}

\renewcommand{\S}{\mathbb{S}}
\newcommand{\T}{T}

\newcommand{\weaksto}{\overset{*}{\rightharpoonup}}
\renewcommand{\H}{\mathcal H}

\theoremstyle{plain} \newtheorem*{theorem*}{Theorem}
\theoremstyle{plain} 
\theoremstyle{plain} \newtheorem*{mthm*}{Main Theorem}
\theoremstyle{plain} \newtheorem*{conjecture*}{Conjecture}
\theoremstyle{plain} 
\theoremstyle{plain} \newtheorem*{problem*}{Problem}

\numberwithin{equation}{section}

\usepackage{color}
\usepackage{graphicx}
\usepackage{tikz}
\usepackage{framed}
\definecolor{shadecolor}{rgb}{0.94, 0.97, 1.0}

\makeatletter
\def\l@subsection{\@tocline{2}{0pt}{2.5pc}{5pc}{}}
\makeatother

\usepackage{pict2e}
\makeatletter
\DeclareRobustCommand{\intprod}{%
  \mathbin{\mathpalette\int@prod{(0.1,0)(0.9,0)(0.9,0.8)}}}
\DeclareRobustCommand{\restrict}{%
  \mathbin{\mathpalette\int@prod{(0.1,0.8)(0.1,0)(0.9,0)}}}	
\newcommand{\int@prod}[2]{%
  \begingroup
  \sbox\z@{$\m@th#1+$}%
  \setlength\unitlength{\wd\z@}%
  \begin{picture}(1,1)
  \roundcap
  \polyline#2
  \end{picture}%
  \endgroup
}
\makeatother

\title[Boundary rectifiability and compactness via $BV$]{Boundary rectifiability and compactness of integral currents
via $BV$ functions}
\author{Giacomo Del Nin}
\address[G.\ Del Nin]{Max Planck Institute for Mathematics in the Sciences, Inselstrasse 22, 04103 Leipzig, Germany}
\email{giacomo.delnin@mis.mpg.de}

\begin{document}
\begin{abstract}
    We present a new proof that an integer rectifiable current with finite mass, and whose boundary has also finite mass, is integral. We deduce the result from De Giorgi's structure theorem for integer-valued $BV$ functions and a cylindrical projection argument. As a consequence, we also give a new proof of the compactness of integral currents that is ultimately based on the $BV$ theory.
\end{abstract}

\maketitle

\setcounter{tocdepth}{1}
\tableofcontents

\section{Introduction}
The space of currents has been widely used as a weak setting to study problems concerning oriented surfaces since the seminal paper by Federer and Fleming \cite{Federer-Fleming}, who expanded the theory of currents introduced by de Rham \cite{DeRham} in a variational direction. Of particular importance for the space of \textit{integral} currents are the compactness theorem and the boundary rectifiability theorem.

The goal of this paper is giving a new proof of the following:

\begin{theorem}[Boundary rectifiability]\label{thm:boundary_rectifiability}
    Let $S$ be an integer rectifiable current in $\R^n$ with finite mass and whose boundary $\partial S$ has finite mass. Then $S$ is integral, i.e., $\partial S$ is also integer rectifiable.
\end{theorem}

As a consequence of this result we will also give a new proof of the following:
\begin{theorem}[Compactness of integral currents]\label{thm:compactness}
    Let $T_j$ be a sequence of integral $k$-currents in $\R^n$, such that 
    \[
    \sup_j \M(T_j)+\M(\partial T_j)<\infty
    \]
    Then, up to subsequence, $T_j\weaksto T$ for some integral $k$-current $T$.
\end{theorem}

One existing proof of Theorem \ref{thm:boundary_rectifiability} passes through the polyhedral deformation theorem and the compactness theorem for integral currents (see \cite{Simon}). White's proof instead relies, via slicing, on the compactness of integral currents, which is proven simultaneously in a mutual induction on the dimension \cite{White}. Fleming gives a proof that is based on Federer's structure theorem characterizing purely unrectifiable sets through projections \cite{Fleming}.

On the other hand, the original proof of Theorem \ref{thm:compactness} by Federer and Fleming \cite{Federer-Fleming} is also based on Federer's structure theorem. Later proofs rely on the ``rectifiable slices theorem'' \cite{Jerrard,White-flat-chains}, see also \cite[Section~8]{Ambrosio-Kirchheim} for a version in metric spaces. Another one is based on a density lemma that ultimately also relies on Federer's structure theorem \cite{Simon}. 
Solomon instead provides a proof that is based on Almgren's multi-valued functions \cite{Solomon}.

Our goal is to give an independent and direct proof of the boundary rectifiability theorem (and as a consequence, also of the compactness theorem) that does not rely on any of these results, but rather relies only on the theory of $\mathbb{Z}$-valued $BV$ functions, through the use of De Giorgi's structure theorem and a projection argument. This is made possible by the fact that, after the projection, a top-dimensional normal current can be identified with a $BV$ function. In particular, we do not employ Federer's structure theorem, the polyhedral deformation lemma, or the rectifiabile slices criterion. 
\begin{figure}
    \centering
    \def\svgwidth{0.4\columnwidth}
    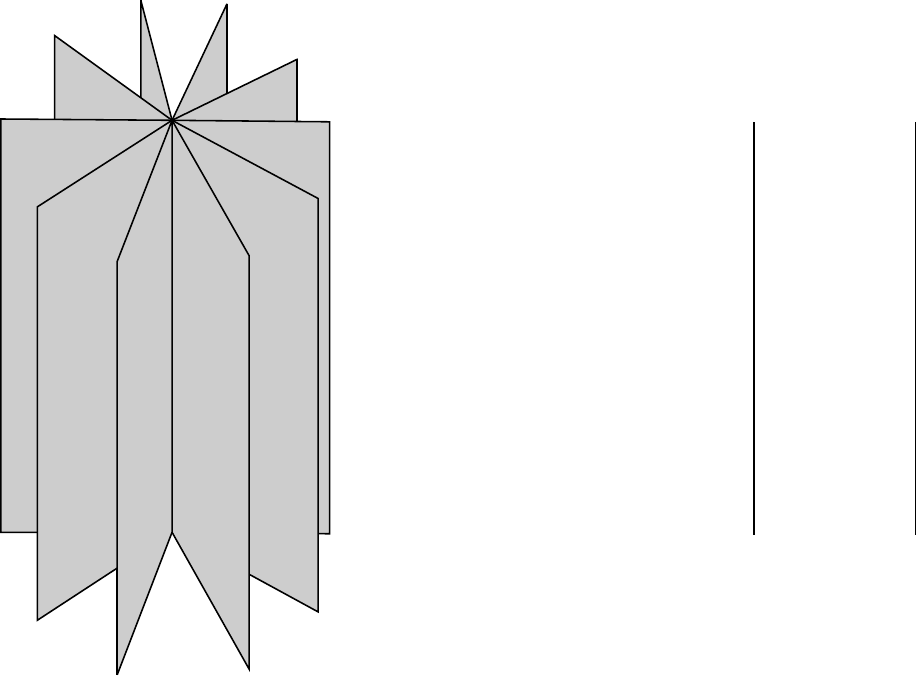

    \caption{ \small The cylindrical projection $f_{x_0,V}$, in the special case where $k=1$, $n=3$ and $V$ is the span of the $1$-vector $\tau_0$. In this case $f_{x_0,V}$ identifies all the half-planes originating at $V$. This is similar to closing an open book and identifying all pages.}
    \label{fig:projection}
\end{figure}
The key novelty of this strategy is using, instead of orthogonal projections onto planes, a ``cylindrical'' projection that is akin to closing the pages of an open book, identifying all the pages (see Figure \ref{fig:projection}). 

\subsection{Outline of the argument}
The rough outline of the boundary rectifiability proof is the following (we refer to Section \ref{sec:preliminaries} for the notation):

\begin{enumerate}
    \item We first identify a ``reduced boundary'' for the $(k+1)$-current $S$, made of all points $x_0\in \supp \|\partial S\|$ of approximate continuity of the orientation $\vec{\partial S}$, for which additionally $\Theta^{*,k}(\|\partial S\|,x_0)<\infty$. By standard results, $\|\partial S\|$-a.e. point satisfies both properties.
    \item For each such $x_0$, we identify a suitable linear $k$-plane $V$ adapted to the local orientation of $\partial S$ (see the next section for the explanation of this choice; morally it should be the tangent space to $\partial S$ at $x_0$) 
    and we consider the cylindrical projection $f_{x_0,V}:\R^n\to\R^{k+1}$ given by
\begin{equation}\label{eq:projection_definition}
    f_{x_0,V}(x):=\big((x-x_0)\cdot v_1,\,\ldots\,,(x-x_0)\cdot v_k, \dist(x-x_0,V)\big),
\end{equation}
where $v_1,\ldots,v_k$ is an orthonormal basis of $V$.
\item We then consider the pushforward $(f_{x_0,V})_*S$: it is an integer rectifiable $(k+1)$-current in $\R^{k+1}$, hence can be identified with a function $u\in BV(\R^{k+1},\Z)$, and its boundary can be identified with $Du$.
If $x_0$ is chosen as in Point (1) it turns out that, thanks to our choice of $V$, the approximate continuity property of the polar is preserved, hence the origin belongs to the ``reduced boundary'' of $u$. A modification of De Giorgi's blowup argument implies the existence of an approximate tangent $k$-plane for $|Du|$ at the origin.
\item By pulling back this information to $\R^n$, we discover that also the current $\partial S$ admits an approximate tangent $k$-plane at $x_0$. Since this holds for $\|\partial S\|$-a.e. point, we deduce the rectifiability of $\partial S$. The integer multiplicity then follows.
\end{enumerate}

We include a simple proof of the case $k=0$ in Section \ref{sec:1-dim_proof}, in order to show the argument in a simplified setting.

Finally, we also deduce compactness of integral currents from the boundary rectifiability.
Instead of arguing by induction on the dimension $k$ (as done for instance in \cite{White}), we go in the other direction, 
employing an induction on the \textit{codimension}, the base case being the top-dimensional one, i.e., compactness of $\Z$-valued $BV$ functions. This provides a proof of the compactness theorem that, except for some general lemmas about currents, is ultimately based only on the theory of $\Z$-valued $BV$ functions.

\subsection{The role of the cylindrical projection}
The new ingredient in the proof of the boundary rectifiability is the use of the cylindrical projection $f_{x_0,V}$ in \eqref{eq:projection_definition}. The main difference with respect to orthogonal projections is that the latter allow for cancellations between far away pieces, thus possibly losing local information about the geometry of $\partial S$. Instead, the projection $f_{x_0,V}$ is capable of preserving local information: observe for instance that it is a proper map, and that the preimage of $B_r(0)$ under $f_{x_0,V}$ is $B_r(x_0)$, while for orthogonal projections it is a whole infinite cylinder. Additionally, and crucially, the preimage of a $k$-dimensional cone about the hyperplane $\{x_{k+1}=0\}$ remains a $k$-dimensional cone, hence allowing the transfer of tangent spaces from $\R^{k+1}$ to $\R^n$.

The heart of the argument is the choice of the appropriate $k$-space $V$ to employ in the projection (see Lemma \ref{lemma:decomposition}).  If the orienting multi-vector $\vec{\partial S}(x_0)$ is simple then $V$ can be taken to be its span. A posteriori $\partial S$ will be rectifiable and thus its orienting vector is simple at $\|\partial S\|$-a.e. point, but we do not know it at this point. Hence we do the next best thing: we consider the best approximation of $\vec{\partial S}$ with a \textit{simple} $k$-vector, and consider the span of the latter. This choice turns out to preserve enough directional information about the boundary orientation to transfer approximate continuity of the polar into the $BV$ setting, allowing for the application of De Giorgi's theorem. The reason is that this projection only cares about the component of $\vec{\partial S}$ parallel to $V$, forgetting about any other ``orthogonal'' components.
%This preservation of approximate continuity is arguably the most delicate part, and most of Section \ref{sec:boundary_rectifiability} is devoted to the proof of this point. 
In the case when we know a priori that $\vec{\partial S}$ is simple (for instance, if $k=1$, or by some other means) the preservation of approximate continuity follows in a straightforward way, hence the argument could be simplified further.

In conclusion, the argument presented here suggests that certain rectifiability questions for currents may systematically reduce to codimension-one rigidity statements after suitable nonlinear projections.

\subsection{De Giorgi's structure theorem}
While De Giorgi's structure theorem for finite perimeter sets is a classical result that can be found in many textbooks on the topic, we could not find any reference to an analogue result for $\mathbb{Z}$-valued $BV$ functions. In particular, we do not just need the rectifiability of $|Du|$ (which follows directly by the coarea formula), but we really need the pointwise statement that the blowup of $|Du|$ at precisely the points we are looking at is flat (see Theorem \ref{thm:blowup}). We thus decided to provide a complete proof of this. 
The main additional difficulty with respect to finite perimeter sets is establishing uniqueness of the blowup  (see Proposition \ref{prop:uniqueness}). In order not to interrupt the flow of the main result, and since we will use De Giorgi's theorem only as a black box, we decided to postpone the proof to the appendix.

\subsection*{Acknowledgements.} I am greatly indebted to Andrea Marchese and Paolo Bonicatto, since the inception of this paper was born out of conversations with them. I am also thankful to Zhuolin Li for many helpful comments on a first draft of the paper.

\subsection*{Addendum}
Upon the completion of this paper I came to know from Stefan Luckhaus that he and Emanuele Spadaro are also working on a method to reduce questions about rectifiability for currents to questions about $BV$ functions. However, up to my understanding, their methods seem to be different.

\section{Preliminaries on currents}\label{sec:preliminaries}

We briefly recall some fundamental notions about currents, mainly to fix the notation. We refer the reader to \cite{KP} and \cite{Federer} for an extensive treatment of currents.

We denote by $\Dcal^k(\R^n)$ the space of compactly supported, smooth differential $k$-forms in $\R^n$.
A $k$-current $T$ in $\R^n$ is a linear, continuous functional on $\Dcal^k(\R^n)$, and we denote this action by $\langle T,\omega\rangle$, $\omega\in \Dcal^k(\R^n)$. The boundary $\partial T$ of $T$ is the $(k-1)$-current defined by $\langle \partial T,\omega\rangle:=\langle T, d\omega\rangle$ for every $\omega\in \Dcal^{k-1}(\R^n)$. A $k$-current $T$ has finite mass if
\[
\M(T):=\sup\{\langle T,\omega\rangle : \, \omega\in\Dcal^k(\R^n), \, |\omega(x)|\le 1\,\forall x\}<\infty.
\]
A $k$-current $T$ is normal if $\M(T)+\M(\partial T)<\infty$. 

If $T$ has finite mass then it can be represented as $T=\tau\mu$, where $\mu$ is a non-negative finite measure and where $\tau\in L^1_\mu(\R^n;\bigwedge_k\R^n)$  is a unit $k$-vector field. In this case we will also denote $\mu$ by $\|T\|$, the mass measure of $T$. Here we denote by $\bigwedge_k\R^n$ the space of $k$-vectors in $\R^n$. A $k$-vector $v\in\bigwedge_k\R^n$ is simple if it can be represented as $v=v_1\wedge \ldots \wedge v_k$ for some vectors $v_1,\ldots, v_k$. We will use $\ibf$ to denote a multi-index: $\ibf=(i_1,\ldots,i_k)$, $i_1<\ldots<i_k$. The space of all $k$-indices over $n$ elements will be denoted by $I(n,k)$. %If both $\Msf(T)<\infty $ and $\Msf(\partial T)<\infty$ then we say that $T$ is normal.

A set $E\subseteq \R^n$ is $\H^k$-rectifiable if it is contained in the union of countably many Lipschitz images of $\R^k$, up to a $\H^k$-negligible set. A $k$-current $T$ is rectifiable if it can be represented as $T=f \tau \H^k\restrict E$, for some $\H^k$-rectifiable set $E$, some multiplicity $f\in L^1_{\H^k}(E)$ and some unit, simple $k$-vector field $\tau(x)=v_1(x)\wedge \ldots \wedge v_k(x)$ with $\{v_i(x)\}_{i=1,\ldots,k}$ spanning the approximate tangent space $\Tan(E,x)$ of $E$ at $x$ for $\H^k$-a.e. $x\in E$.

A $k$-current $T$ is integer rectifiable if in addition to being rectifiable the multiplicity $f$ is integer-valued. $T$ is called integral if it is normal, integer rectifiable, and if also its boundary $\partial T$ is integer rectifiable.

\subsection{A pushforward formula}

The classical formula for the pushforward of a $k$-current with finite mass $T=\tau\mu$ in $\R^n$ under a proper $C^1$ map $f:\R^n\to\R^m$ reads
\[
\langle f_*T,\omega\rangle:=\langle T, f^*\omega\rangle=\int \langle \omega(f(x)), df|_x[\tau(x)]\rangle\, d\mu(x).
\]
Here, for a linear map $L$, we denote by $L[\tau]$ the pushforward of the $k$-vector $\tau$, defined by
\[
L[v_1\wedge\ldots\wedge v_k]:=L[v_1]\wedge\ldots \wedge L[v_k]
\]
on simple $k$-vectors and extended by multi-linearity. In the following we will need an analogous formula for the pushforward under the projection $f=f_{x_0,V}$, which is only Lipschitz. We could appeal to the results in \cite{Alberti-Marchese}, which are much more general (see e.g. \cite[Proposition~5.17]{Alberti-Marchese}), but in our specific case we can provide a sketch of the proof.

\begin{lemma}[Pushforward under cylindrical projection]\label{lemma:pushforward}
    Let $T=\tau\mu$ be a normal $k$-current in $\R^n$, and $f=f_{x_0,V}:\R^n\to \R^{k+1}$ the cylindrical projection defined in \eqref{eq:projection_definition}. Then 
    \[
    \langle f_*T,\omega\rangle=\int \langle \omega(f(x)), df|_x[\tau(x)]\rangle\, d\mu(x),
    \]
    where $df|_x[\tau(x)]$ is defined classically for $x\in\R^n\setminus (x_0+V)$, while for $x\in x_0+V$ 
    \[
    df|_x[\tau(x)]:=(\pi_V(\tau^V(x)),0),    \]
    where $\tau^V(x)$ denotes the component of $\tau$ parallel to $V$, and $\pi_V$ the orthogonal projection onto $V$ (identified with $\R^{k}$).
\end{lemma}

\begin{proof}[Sketch of proof]
    We approximate $f$ with $f_\eps(x):=(\pi_V(x-x_0),\sqrt{\eps^2+\dist(x,x_0+V)^2})$. Applying the standard pushforward formula and passing to the limit thanks to Lebesgue's dominated convergence we reach the conclusion.
\end{proof}

\subsection{Isoperimetric inequality}
We recall the following simple isoperimetric inequality for currents (see \cite[7.4.4]{KP}):

\begin{lemma}[Cone isoperimetric inequality]\label{lemma:isoperimetric_inequality}
    Let $T$ be an integral $k$-current in $\R^n$, $1\le k\le n-1$, supported in $B(0,R)$ and such that $\partial T=0$. Then there exists an integral $(k+1)$-current $S$, supported in $B(0,R)$, such that $\partial S=T$ and for some dimensional constant $C$
    \[
    \M(S)\le CR\M(T).
    \]
\end{lemma}

\section{Boundary rectifiability}\label{sec:boundary_rectifiability}

\subsection{Proof for $1$-currents}\label{sec:1-dim_proof}

We give the proof of the $1$-dimensional case first (i.e., $k=0$), as the argument is very straightforward and does not require additional results, but already contains the essence of the argument. The higher-dimensional proof will follow the same philosophy, but additional work is needed to show the preservation of approximate continuity under projection.

\begin{theorem}[$1$-dimensional case]
    Let $S$ be an integer rectifiable $1$-current in $\R^n$, with $\Msf(S)+\Msf(\partial S)<\infty$. Then $\partial S$ is supported on a finite set, and it has integer multiplicity.
\end{theorem}

\begin{proof}
    Let $\mu$ denote the (signed) measure $\partial S$. We can write $\mu=\nu|\mu|$, for some $\nu$ attaining only values $\pm 1$ at $|\mu|$-a.e. point. Let us fix a strong Lebesgue point $x_0$ for $\nu$ with respect to $\mu$, i.e. $x_0\in\supp |\mu|$  such that $|\nu(x_0)|=1$ and
    \[
    \lim_{r\to 0} \fint_{B_r(x)} |\nu(y)-\nu(x_0)|\,d|\mu|(y)=0.
    \]
    By standard differentiation theorems, the measure $\mu$ is concentrated on the set of such points.
    We consider the projection map $f:\R^n\to \R$ given by
    \[
    f(x):=|x-x_0|,
    \]
    and define the pushforward measure $f_\# \mu$. Observe that $f$ is $1$-Lipschitz and proper and attains only non-negative values, hence $f_*S$ is an integer rectifiable $1$-current with boundary of finite mass supported in $[0,\infty)$, thus can be identified with a $\mathbb{Z}$-valued $BV$ function $u$. Moreover $\partial(f_*S)=f_*(\partial S)=f_\#\mu$ can be identified with $Du$. We claim that $0$ belongs to the support of $f_\#\mu$. Indeed,
    \begin{equation}\label{eq:1dim}
        f_\# \mu(B_r(0))=\mu(f^{-1}(B_r(0)))=\mu(B_r(x_0))=\int_{B_r(x_0)} \nu(y) d|\mu|(y)
    \end{equation}
    and thanks to the Lebesgue property this is non-zero for all sufficiently small $r>0$.
    Since $f_\#\mu=Du$, with $u\in BV(\R,\Z)$, we know that $Du$ is purely jump and thus $0$ must be a jump point for $u$, and as a consequence $|f_\#\mu|\ge \delta_0$. This implies that $|\mu|(B_r(x_0))\ge 1$. Since $\mu$ is a finite measure, this can happen only at finitely many points, and we deduce that the support of $\|\partial S\|$ is finite. Moreover, $\partial S$ has integer multiplicity, since for every $x_0$ in the support of $\partial S$ the expression in \eqref{eq:1dim} converges to $Du(\{0\})$, which is an integer.
\end{proof}

 The extension of this strategy to higher-dimensional currents requires to use the cylindrical projection $f_{x_0,V}$ defined in \eqref{eq:projection_definition} instead of the distance function, and to rely on Theorem \ref{thm:blowup} instead of the structure of one-dimensional $BV$ functions.

\subsection{Adapted projection lemma}

The following result is in a way the key to the preservation of approximate continuity. It tells us that for any $k$-vector $\tau$ we can find an orthonormal basis $e_1,\ldots,e_n$ of $\R^n$ such that $\tau$ has a non-trivial component along $e_1\wedge\ldots\wedge e_k$, and such that the only other non-trivial components, when written in the basis $\{e_{i_1}\wedge\ldots\wedge e_{i_k}\}$, are those that are missing at least \textit{two} elements of $\{e_1,\ldots,e_k\}$. In particular, if $V=\mathrm{span}\{e_1,\ldots, e_k\}$, all these additional components are sent to zero by the cylindrical projection $f_{x_0,V}$. The way to choose $e_1,\ldots, e_k$ is conceptually very simple: we choose the best approximation of $\tau$ with a simple $k$-vector.

\begin{lemma}[Adapted projection]\label{lemma:decomposition}
    Let $\tau\in \Wedge_k(\R^n)$ be an arbitrary $k$-vector. Then there exists an orthonormal basis $e_1,\ldots,e_n$ of $\R^n$ such that
    \begin{equation}\label{eq:decomposition_tau}
    \tau =\tau_{1,\ldots,k}\, e_1\wedge\ldots \wedge e_k+\sum_{\ibf\in I(n,k): \, \#(\ibf\cap\{1,\ldots,k\})\le k-2} \tau_\ibf e_\ibf ,\qquad \tau_{1,\ldots,k}\ne 0.
    \end{equation}
    In particular, with $V:=\mathrm{span}\{e_1,\ldots ,e_k\}$,
    \begin{equation}\label{eq:projection_of_tau}
        d (f_{x_0,V})|_x [\tau]=\tau_{1,\ldots,k}e_1\wedge \ldots\wedge e_k\qquad\text{for every $x\in\R^n$,}
    \end{equation}
    where $(df_{x_0,V})|_x$ is defined as in Lemma \ref{lemma:pushforward}.
\end{lemma}

\begin{proof}
    Consider the following problem:
    \[
    \inf\{|\tau-v|:\, \text{$v$ simple $k$-vector}\}.
    \]
    It is clear that, by compactness, the  minimum is achieved by some non-zero simple vector $v=v_1\wedge\ldots\wedge v_k$. Observe that $0$ can not be a minimum because
    \[
    |\tau-0|=|\tau|=\sqrt{\sum \tau_{\ibf}^2}>|\tau-\tau_{\ibf_0}e_{\ibf_0}|
    \]
    for any choice of $\ibf_0$ with $\tau_{\ibf_0}\ne 0$.
    We can also find an orthonormal basis $\{e_1,\ldots,e_n\}$ such that $v$ is a multiple of $e_1\wedge\ldots\wedge e_k$. We claim that this basis is the one we are looking for. Writing 
    %First we can write $\tau$ in this basis as
    %\[
    $\tau=\sum_\ibf \tau_\ibf e_\ibf$
    %\]
    in this basis, observe that necessarily $v=\tau_{\ibf_0}e_{\ibf_0}$, with $\ibf_0=(1,\ldots,k)$ (since otherwise we would find a better competitor) and that
    \[
    |\tau-v|^2=\sum_{\ibf\ne \ibf_0} \tau_{\ibf}^2.
    \]
    Suppose now by contradiction that $\tau_{\ibf_1}\ne 0$ for some $\ibf_1$ with $\#(\ibf_1\cap \{1,\ldots,k\})= k-1$. We can assume, up to a relabeling, that $\ibf_1=\{1,\ldots, k-1,k+1\}$. Then we can write
    \begin{align*}
        \tau&=\tau_{\ibf_0} e_1\wedge \ldots \wedge e_k+\tau_{\ibf_1}e_1\wedge\ldots\wedge e_{k-1}\wedge e_{k+1}+\sum_{\ibf\not\in\{\ibf_0,\ibf_1\}%\substack{\ibf\ne \ibf_0\\\ibf\ne \ibf_1}
        }\tau_{\ibf}e_{\ibf}\\
        &=  \underbrace{e_1\wedge \ldots \wedge e_{k-1}\wedge (\tau_{\ibf_0}e_k+\tau_{\ibf_1} e_{k+1})}_{\qquad=:w}+\sum_{\ibf\not\in\{\ibf_0,\ibf_1\}%\substack{\ibf\ne \ibf_0\\\ibf\ne \ibf_1}
        }\tau_{\ibf}e_{\ibf}
    \end{align*}
    where clearly the $k$-vector $w$ is simple. We now observe that
    \[
    |\tau-w|^2=\sum_{\substack{\ibf\ne \ibf_0\\\ibf\ne \ibf_1}}\tau_{\ibf}^2<\sum_{\ibf\ne \ibf_0}\tau_{\ibf}^2=|\tau -v|^2
    \]
    because $\tau_{\ibf_1}\ne 0$. This shows that $v$ was not a minimizer, and thus gives a contradiction. It follows that $\tau_{\ibf}= 0$ for every $\ibf$ with $\#(\ibf\cap \{1,\ldots,k\})= k-1$, hence the thesis.    

    We are left to show \eqref{eq:projection_of_tau}. If $x\in x_0+V$ the equality follows by definition. If instead $x\not\in x_0+ V$, it is sufficient to observe that for every $i\in\{k+1,\ldots,n\}$
    \[
    d(f_V)|_x[e_i]=e_{k+1},
    \]
    therefore all terms appearing in the sum in \eqref{eq:decomposition_tau}, when projected to $\R^{k+1}$, have a repeating factor $e_{k+1}$, thus are zero.
\end{proof}

\subsection{Mass comparison under projection}
Given a $1$-Lipschitz map $f$, one always has $\|f_*T\|(\R^n)\le \|T\|(\R^n)$, but in principle there could be a large drop of the total mass. The next lemma shows that, if $f$ is the projection defined in \eqref{eq:projection_definition}, then this drop is controlled, at least locally, around approximate continuity points of the orienting vector field.

\begin{lemma}[Mass comparison]\label{lemma:mass_comparison}
    Let $T=\tau\mu$ be a normal $k$-current, and let $x_0$ be a Lebesgue point of $\tau$, with value $\tau_0=\tau(x_0)$. Set
    \[
    \omega(r):=\fint_{B_r(x_0)}|\tau(x)-\tau_0|d\mu(x),
    \]
    so that $\omega(r)\to 0$ as $r\to 0$.
    Let $V$ be any $k$-plane, and $f=f_{x_0,V}$ the corresponding projection map defined in \eqref{eq:projection_definition}. Then for every $r>0$ and every open subset $\Omega\subset B_r(0)$,
    \begin{equation}\label{eq:mass_comparison}
    \|T\|(f^{-1}(\Omega))\ge \|f_*T\|(\Omega)\ge |\tau_0^V| \|T\|(f^{-1}(\Omega))-\|T\|(B_r(x_0))\omega(r).
    \end{equation}
    In particular,
    \begin{equation}\label{eq:mass_comparison_ratio}
        \liminf_{r\to 0}\frac{\|f_*T\|(B_r(0))}{\|T\|(B_r(x_0))}\ge |\tau_0^V|.
    \end{equation}
\end{lemma}

\begin{proof}
    The left-most inequality in \eqref{eq:mass_comparison} follows from the $1$-Lipschitzianity of $f$, so we only need to prove the right-most inequality. We can suppose without loss of generality that $|\tau_0^V|>0$, otherwise there is nothing to prove, and we assume for simplicity that $\tau_0^V >0$.
    
    Fix any $\phi\in C^\infty_c(\Omega)$ with $\|\phi\|_{C^0}\le 1$. Then
    \begin{align*}
        \|f_*T\|(\Omega)&\ge \langle f_*T,\phi dx_1\wedge\ldots\wedge dx_k\rangle\\
        &= \langle T,\phi\circ f dx_1\wedge\ldots\wedge dx_k\rangle\\
        &= \int \phi(f(x))\langle dx_1\wedge\ldots\wedge dx_k,\tau(x)\rangle d\|T\|(x)\\
        &=\int \phi(f(x))\langle dx_1\wedge\ldots\wedge dx_k,\tau_0\rangle d\mu(x)\\
        &\qquad +\int \phi(f(x))\langle dx_1\wedge\ldots\wedge dx_k,\tau(x)-\tau_0\rangle d\|T\|(x).
    \end{align*}
    We can estimate the modulus of the second integral by
    \[
    \|\phi\|_{C^0}\int |\tau(x)-\tau_0|d\|T\|(x)\le \omega(r) \|T\|(B_r(x_0)).
    \]
    Instead the first integral coincides with
    \[
    \int \phi(f(x))\tau_0^V d\|T\|(x).
    \]
    By taking the supremum among all $\phi$ with support in $\Omega$ and modulus at most 1 we find
    \[
    \|f_*T\|(\Omega)\ge \tau_0^V \|T\|(f^{-1}(\Omega))-\omega(r) \|T\|(B_r(x_0)).
    \]
    This proves \eqref{eq:mass_comparison}. To prove \eqref{eq:mass_comparison_ratio} it is sufficient to apply the previous inequality with $\Omega=B_r(0)$, observing that $f^{-1}(B_r(0))=B_r(x_0)$.
\end{proof}

We can immediately deduce two useful corollaries, connecting density and geometric properties of $f_*T$ and $T$.

\begin{corollary}[Density comparison]\label{cor:density_comparison}
    Let $T$, $x_0$, $\tau_0$ and $\omega$ be as in Lemma \ref{eq:mass_comparison}, with $\tau_0^V\ne 0$. Then
    \[
    \frac{\|T\|(f^{-1}(\Omega))}{\|T\|(B_r(x_0))}\le \frac{1}{|\tau_0^V|} \frac{\|f_*T\|(\Omega)}{\|f_*T\|(B_r(0))}+\frac{1}{|\tau_0^V|}\omega(r).
    \]
\end{corollary}

\begin{proof} Using Lemma \ref{lemma:mass_comparison} and that $\|f_*T\|(B_r(0))\le \|T\|(B_r(x_0))$ we deduce
    \begin{align*}
    \frac{|\tau_0^V|\|T\|(f^{-1}(\Omega))}{\|T\|(B_r(x_0))}& \le \frac{\|f_*T\|(\Omega)+\|T\|(B_r(x_0))\omega(r)}{\|T\|(B_r(x_0))}     \le \frac{\|f_*T\|(\Omega)}{\|f_*T\|(B_r(x_0))}+\omega(r).\qedhere
    \end{align*}
\end{proof}

Before stating the next corollary, let us recall the following definition: given $x\in \R^n$, $L>0$ and a $k$-plane $V$ we define the cone
 \[
 X(x,V,L):=\{y\in\R^n:\, |\pi_{V^\perp}(y-x)|\le L|\pi_V(y-x)|\}.
 \]

\begin{corollary}[Tangent cones under projection]\label{cor:cones_under_projection}
    Let $T=\tau\|T\|$ be a normal $k$-current, and let $x_0$ be a Lebesgue point of $\tau$, with value $\tau_0=\tau(x_0)$. Let $V$ be any $k$-plane for which $\tau_0^V\ne 0$, and let $f=f_{x_0,V}$ be the associated cylindrical projection as in \eqref{eq:projection_definition}. Assume that the current $f_*T$ admits a tangent cone about $f(V)$ at the origin, in the sense that
    \[
    \lim_{r\to 0}\frac{\|f_*T\| (B_r(0)\setminus X(0,f(V),L))}{\|f_*T\|(B_r(0))}=0
    \]
    for some $L$. Then $T$ admits a tangent cone about $V$ at $x_0$, in the sense that
    \begin{equation}\label{eq:tangent_cone}
    \lim_{r\to 0}\frac{\|T\| (B_r(x_0)\setminus X(x_0,V,L))}{\|T\|(B_r(x_0))}=0.
    \end{equation}
\end{corollary}

\begin{proof}
    From the definition it is straightforward to check that $f$ preserves cones, and more precisely
    \[
    f^{-1}(B_r(0)\setminus X(0,f(V),L))=B_r(x_0)\setminus X(x_0,V,L).
    \]
    From Corollary \ref{cor:density_comparison}, applied with $\Omega=B_r(0)\setminus X(0,f(V),L)$, we obtain
    \[
    \frac{\|T\|(B_r(x_0)\setminus X(x_0,V,L))}{\|T\|(B_r(x_0))}\le \frac{1}{|\tau_0^V|} \frac{\|f_*T\|(B_r(0)\setminus X(0,V,L))}{\|f_*T\|(B_r(0))}+\frac{1}{|\tau_0^V|} \omega(r).
    \]
    By assumption the right-hand side goes to zero as $r\to 0$, hence the conclusion.
\end{proof}

\subsection{Approximate continuity under projection}
Our next goal is showing that the approximate continuity property is preserved under the cylindrical projection $f_{x_0,V}$, if $V$ is chosen as in Lemma \ref{lemma:decomposition}. The point is that, with this choice, the projection removes exactly the orthogonal components of the orientation that could otherwise create corners or oscillations in the image current. %We start with a simple lemma.

\begin{lemma}[Approximate continuity under projection]\label{lemma:preservation_Lebesgue}
    Let $T=\tau\|T\|$ be a normal $k$-current, with $\tau$ unit $k$-vector at $\mu$-a.e. point. Let $x_0\in\supp T$ be such that
    \[
    \lim_{r\to 0} \fint_{B_r(x_0)} |\tau(x)-\tau_0|\, d\|T\|(x)=0,
    \]
    where $\tau_0=\tau(x_0)$ is a unit $k$-vector.
    Let $V$ be the $k$-plane given by Lemma \ref{lemma:decomposition} applied to $\tau_0$, and let $f=f_{x_0,V}$ be the corresponding projection defined in \eqref{eq:projection_definition}. Then $0$ is a Lebesgue point of $f_* T$, namely, writing $f_*T=\tilde \tau \|f_*T\|$, it holds
    \begin{equation}
        \lim_{r\to 0} \fint_{B_r(0)} |\tilde \tau(y)-\tilde\tau(0)|\, d\|f_*T\|(y)=0.
    \end{equation}
\end{lemma}

\begin{proof}
    The idea is to compare the projected current with the current obtained by projecting the frozen orientation $\tau_0$, showing that they are infinitesimally equivalent.
    Let us write
    \[
    \tau_0=\tau_0^V e_1\wedge\ldots\wedge e_k+\tilde \tau_0, \qquad\text{$\tilde\tau_0\perp e_1\wedge \ldots\wedge e_k$},
    \]
    and let us assume without loss of generality that $\tau_0^V>0$. 
    We claim that the currents $\tilde T:=f_* T=\tilde\tau\tilde\mu$ and $T':=(f_*\tau_0)\tilde\mu=\tau_0^V e_1\wedge\ldots\wedge e_k\, \tilde \mu$ satisfy
    \begin{equation}\label{eq:almost_Lebesgue}
    \lim_{r\to 0}\frac{\|T'-\tilde T\|(B_r(0))}{\tilde \mu(B_r(0))}=0.
    \end{equation}
    To show this we start from 
    \begin{align*}
        \|T'-\tilde T\|(B_r(0)) &= \sup \{\langle T'-\tilde T,\omega\rangle :\, \omega\in \D^k(B_r(0)),\,\|\omega\|\le 1\}.
    \end{align*}
    Let us fix $\omega$ with support in $B_r(0)$ and let us expand $\langle T',\omega\rangle $ and $\langle \tilde T,\omega\rangle $ separately. Using Lemma \ref{lemma:pushforward}, and the fact that $df|_x$ is a linear map for every $x$, we obtain
    \begin{align*}
        \langle \tilde T,\omega\rangle &= \langle f_*T,\omega\rangle \\
        &=\int \langle \omega(f(x)),df|_x \tau(x)\rangle \, d\|T\|(x)\\
        & =\int \langle \omega(f(x)),df|_x [\tau(x)-\tau_0]\rangle \, d\|T\|(x)+\int \langle \omega(f(x)),df|_x \tau_0\rangle \, d\|T\|(x).
    \end{align*}
    Using that $df|_x$ is $1$-Lipschitz, the modulus of the first integral is estimated by
    \[
    \|\omega\|_{C^0} \int_{B_r(x_0)} |\tau(x)-\tau_0|\, d\|T\|(x).
    \]
    The second integral instead coincides with $\T'$, since $df|_x\tau_0=\tau_0^V$ for every $x\in\R^n$, by virtue of \eqref{eq:projection_of_tau} (this is the only point where it is crucial to use the specific plane $V$ given by Lemma \ref{lemma:decomposition}, see Remark \ref{rmk:projection}).
    Hence
    \begin{align*}
        |\langle T'-\tilde T,\omega \rangle|\le \|\omega\|_{C^0}\int_{B_r(x_0)} |\tau(x)-\tau_0|\, d\|T\|(x).
    \end{align*}
    Taking the supremum among all $\omega\in \D^k(B_r(0))$ with $\|\omega\|_{C^0}\le 1$, and using \eqref{eq:mass_comparison_ratio}, we find that
    \[
    \frac{\|\tilde T-T'\|(B_r(0))}{\tilde \mu(B_r(0))}\le\frac{\mu(B_r(0))}{ \tilde\mu(B_r(x_0))}\fint_{B_r(x_0)} |\tau(x)-\tau_0|\, d\|T\|(x) \to 0.
    \]
    This proves \eqref{eq:almost_Lebesgue}. Observe that this is almost the desired Lebesgue property, with $\tilde\tau(0)=\tau_0^V e_1\wedge\ldots\wedge e_k$. The only difference is that the measure $\tilde \mu$ is only a multiple of the mass measure of $f_*T$, and $\tilde \tau$ is not necessarily unit. To reach the conclusion it is sufficient to apply the following elementary lemma.
\end{proof}

\begin{lemma}[Unit vs. non-unit approximate continuity]\label{lemma:almost_Lebesgue}
    Let $T=v\sigma$ be a $k$-current (with $v$ not necessarily unit) and suppose that, for some $k$-vector $v_0\ne 0$,
    \[
    \lim_{r\to 0}\fint_{B_r(x_0)} |v(x)-v_0|\, d\sigma(x)=0.
    \]
    Then $x_0$ is a Lebesgue point of $T$, and more precisely
    \[
    \lim_{r\to 0}\fint_{B_r(x_0)} \left|\frac{v(x)}{|v(x)|}-\frac{v_0}{|v_0|}\right|\, |v(x)| d\sigma(x)=0.
    \]
\end{lemma}

\begin{proof}
    By triangle inequality
    \begin{align*}
        \left|\frac{v(x)}{|v(x)|}-\frac{v_0}{|v_0|}\right|\, |v(x)|& =\frac{1}{|v_0|} \big|v(x)|v_0|-v_0|v(x)| \big|\\
        &= \frac{1}{|v_0|} \big|v(x)|v_0|-v_0|v_0|+v_0|v_0|-v_0|v(x)| \big|\\
        &\le \frac{1}{|v_0|} |v_0|\,|v(x)-v_0|+\frac{1}{|v_0|}|v_0|\,\big||v_0|-|v(x)| \big|\\
        &\le \left(\frac{1}{|v_0|}+1\right)|v(x)-v_0|.
    \end{align*}
    The thesis follows by integrating in $\sigma$.
\end{proof}

\begin{remark}\label{rmk:projection}
    For the conclusion of Lemma \ref{lemma:preservation_Lebesgue} to hold, it is crucial that the plane $V$ is chosen so that \eqref{eq:projection_of_tau} is satisfied. For instance, consider $T$ to be a $1$-current that, in the ball $B_1(0)$, coincides with $v\H^1\restrict \ell_v$ for some $v$, where $\ell_v$ is the span of $v$. If $V$ is chosen to be any line different from $\ell_v$, then the projected current $(f_V)_*T$ has a corner at the origin, and can not have an approximate continuity point there.
\end{remark}

\subsection{Proof of the boundary rectifiability}

We write $\partial S=T=\tau \|T\|$, with $\tau$ unit $k$-vector field. Let us consider a point $x_0\in \supp(T)$ such that
\begin{equation}\label{eq:upper_density_T}
    \Theta^{*,k}(\|T\|,x_0)<\infty
\end{equation}
and
\[
\lim_{r\to 0}\fint_{B_r(x_0)} |\tau(x)-\tau_0|\, d\|T\|(x)=0
\]
for some unit $k$-vector $\tau_0$. By standard differentiation theorems, and since $\|T\|\ll \H^k $ for every normal $k$-current $T$ (see Chapter 6, Theorem~2.43 in \cite{Simon}), $\|T\|$-a.e. point satisfies both properties. Lemma \ref{lemma:decomposition} provides us with an orthonormal basis $e_1,\ldots,e_n$ (that we will assume without loss of generality to be the standard basis of $\R^n$) such that \eqref{eq:decomposition_tau} is satisfied.

Let us set $V:=\mathrm{span}\{e_1,\ldots,e_k\}$, and let us consider the cylindrical projection $f=f_{x_0,V}:\R^n\to \R^{k+1}$ defined by
\[
f(x)=\big( (x-x_0)\cdot e_1,\ldots, (x-x_0)\cdot e_k, \dist(x,x_0+V)\big).
\]
Observe that $f$ is $1$-Lipschitz, and thus $\tilde S:=f_*S$ is a well-defined integer multiplicity normal $(k+1)$-current in $\R^{k+1}$, supported on $\R^k\times [0,\infty)$. Hence $\tilde S$ is represented by a $\Z$-valued $BV$ function $u$:
\[
\tilde S= u e_1\wedge\ldots \wedge e_{k+1} \Lcal^{k+1},\qquad u\in BV(\R^{k+1},\Z), \quad \supp (u)\subseteq \R^k\times[0,\infty).
\]
Moreover, $\tilde T:=f_*T$ is a well-defined normal $k$-current in $\R^{k+1}$, such that $\partial \tilde S=\tilde T$, and thus is identified with $Du=\nu_u|Du|$:
\[
f_* T=\tilde \tau|Du|\qquad\text{and}\qquad \|f_* T\|=|Du|,
\]
where $\tilde\tau=v_1\wedge\ldots\wedge v_k$, with $v_1,\ldots, v_k$ spanning $\nu_u^\perp$, and such that $\tilde\tau\wedge \nu_u=e_1\wedge\ldots\wedge e_{k+1}$.
By Lemma \ref{lemma:preservation_Lebesgue}, $0$ is an approximate continuity point of $f_* T$, or equivalently, an approximate continuity point of $\nu_u=\frac{Du}{|Du|}$ with respect to $|Du|$. 
Moreover, since by $1$-Lipschitzianity $\|f_*T\|(B_r(0))\le\|T\|(B_r(x_0))$, from \eqref{eq:upper_density_T}  we deduce also that
\[
\Theta^{*,k}(|Du|,0)=\Theta^{*,k}(\|f_*T\|,0)\le \Theta^{*,k}(\|T\|,x_0)<\infty.
\]
We are now in the position to apply De Giorgi's Theorem \ref{thm:blowup} to $u$. We find that
\[
\frac{1}{r^k}(Z_{0,r})_\#\|f_*T\|=\frac{1}{r^k}(Z_{0,r})_\#|Du|\weaksto m\H^{n-1}\restrict f(V)
\]
for some $m\in\N\setminus\{0\}$, and moreover that $\Theta^k(\|f_*T\|,0)=m$. In particular, again by 1-Lipschitzianity, this implies that
\begin{equation}\label{eq:positive_lower_density}
\Theta_*^k(\|T\|,x_0)\ge \Theta_*^k(\|f_*T\|,0)=m.
\end{equation}
It follows that, for every $L>0$, $X(0,f(V),L)$ is an approximate tangent cone at $0$ for $\|f_* T\|=|Du|$. By Corollary \ref{cor:cones_under_projection}, $X(x_0,V,L)$ is an approximate tangent cone of $\|T\|$ at $x_0$ in the sense of \eqref{eq:tangent_cone}, and taking into account \eqref{eq:positive_lower_density} this means that
\[
\lim_{r\to 0}\frac{\|T\| (B_r(x_0)\setminus X(x_0,V,L))}{r^k}=0.
\]
As this holds for $\|T\|$-almost every point, we conclude that $T$ is $\H^k$-rectifiable, hence of the form $T=g\tau \H^k\restrict E$, for some $\H^k$-rectifiable set $E$, some unit $k$-vector field $\tau$ orienting $E$, and some multiplicity $g$ that we can assume to be non-negative.

We are only left to show that $g$ is integer-valued $\|T\|$-almost everywhere. Now we can apply again Lemma \ref{lemma:mass_comparison}, but with the gained knowledge that for $\|T\|$-almost every $x_0$ we can take $\tau_0^V=\tau(x_0)$, because $\tau$ is simple $\|T\|$-almost everywhere. Hence we deduce
\[
\lim_{r\to 0} \frac{\|f_*T\|(B_r(0))}{\|T\|(B_r(x_0))}=1.
\]
On the other hand, by Theorem \ref{thm:blowup} we know that $\Theta^k(\|f_*T\|,0)=m$, and this implies that $\Theta^k(\|T\|,x_0)=m\in \mathbb{N}$, 
whence the density is integer-valued $\|T\|$-almost everywhere. \hfill\qedsymbol

\section{Compactness of integral currents}\label{sec:compactness}

We now come to the proof of Theorem \ref{thm:compactness}, showing how compactness follows from the boundary rectifiability theorem, proceeding by induction on the \textit{codimension}. The base case $k=n$ is simply the compactness for $\Z$-valued $BV$ functions. Suppose then that $1\le k\le n -1$ and that $T_j\weaksto T$. The idea is the following:
\begin{enumerate}
    \item If $\partial T_j=0$ then we can use the isoperimetric inequality to find $(k+1)$-currents $S_j$ with equibounded masses and such that $\partial S_j=T_j$. By the induction hypothesis $S_j$ converge to an integral current $S$, hence also $T=\partial S$ is integral and we are done.
    \item If $\partial T_j\ne 0$, then we convert them into boundaryless currents by adding a cylindrical correction term in one higher dimension and projecting it back through a slanted projection (see Figure \ref{fig:cylinder}). We can then apply a similar induction argument to show that the ``filled-in'' currents converge to an integral current, and finally exploit the directional freedom in the slanted projection to rule out cancellations between the original and the added parts. This shows that $T$ itself is an integer rectifiable current, hence by the boundary rectifiability also integral.
\end{enumerate}
We now turn to the actual proof.

\subsection{Some reductions}
We are going to give the proof under the additional assumption that there exists an $R>0$ such that
\[
\supp (T_j)\subseteq B(0,R)\qquad \text{for every $j$.}
\]
The general case can be reduced to this for instance by projecting all currents $T_j$ to $B(0,R)$, applying the compactness in this case, and sending $R\to\infty$ to deduce compactness of the original sequence, since the projection coincides with the identity inside $B(0,R)$. Alternatively, one can adapt the proof below with minor modifications. 
We can also assume, using the compactness theorem for normal currents, that $T_j$ already converge to some normal $k$-current $T$. Moreover, we restrict to $k\ge 1$ as the case $k=0$ is elementary.\\

\subsection{Proof of the compactness theorem}
Let $1\le k\le n$.\\

\underline{\textit{Case $k=n$}}. In this case the statement follows from the compactness theorem for $\Z$-valued $BV$ functions. Indeed, $T_j$ can be identified with a $\Z$-valued $BV$ function $u_j$, for which $\Msf(T_j)=\|u\|_{L^1}$, $\Msf(\partial T_j)=|Du_j|(\R^n)$, thus
\[
\|u_j\|_{L^1}+|Du_j|(\R^n)\le C.
\]
The standard compactness in $BV$ (see \cite[Theorem~3.23]{AFP}) yields that, up to subsequences, $u_j\to u$ in $L^1_{loc}$ for some $u\in BV$, and hence that $T_j\weaksto T$ where $T$ is the $k$-current associated with $u$. Passing to a subsequence converging almost everywhere, we discover that $u$ must also be $\Z$-valued, i.e., that $T$ is integral.\\

\underline{\textit{Case $k+1\implies k$}}. Assume now that we have proven the theorem for sequences of integral $(k+1)$-currents. We want to show that it also holds for $k$-currents. The boundaryless case follows almost directly from the inductive assumption: if $\partial T_j=0$, then we use the isoperimetric inequality given by Lemma \ref{lemma:isoperimetric_inequality} to obtain a sequence of integral $(k+1)$-currents $S_j$, supported in $B(0,R)$, with $\partial S_j=T_j$, and such that
\[
\sup_j \Msf(S_j)+\Msf(\partial S_j)<\infty.
\]
Applying the compactness theorem in dimension $k+1$ we obtain that, up to subsequences, 
\[
S_j\weaksto S
\]
for some integral $(k+1)$-current $S$. Therefore $\partial S$ is integral as well. It follows that
\[
T_j=\partial S_j\weaksto \partial S=T,
\]
which shows the conclusion.

In the general case, we only need to find a way to ``fill the holes'' to make the $T_j$'s boundaryless, applying the same argument in one higher dimension, and argue that the added parts do not interfere in the limiting process, thus concluding that $T$ is integer rectifiable. To this aim, let us consider the following cylindrical $k$-currents in $\R^n\times \R$ (see Figure \ref{fig:cylinder}):
\begin{equation}\label{eq:boundary_explicit}
\partial (T_j\times \curr{0,1})=\partial T_j \times\curr{0,1} +T_j\times (\delta_1-\delta_0).
\end{equation}
Given a vector $v\in \R^n$, we define the slanted projection $\pi_v:\R^n\times \R\to \R^n$
\[
\pi_v((x,x_{n+1})):=x+vx_{n+1},\qquad x\in\R^n,\, x_{n+1}\in \R.
\]
Observe that $\pi_v$ is the identity on the horizontal slice $\R^n\times\{0\}$, and sends the vector $(0^n,1)$ to $v$.
For a fixed $v\in\R^n$, we now consider the sequence of integral $k$-currents
\begin{equation}\label{eq:tildeT_j_definition}
\widetilde T_j:=(\pi_v)_*(\partial (T_j\times \curr{0,1})).
\end{equation}
Observe that $\partial \widetilde T_j=0$, and that
\[
\M(\widetilde T_j)\le |v|\M(\partial T_j)+2\M(T_j).
\]
By the isoperimetric inequality of Lemma \ref{lemma:isoperimetric_inequality} there exist integral $(k+1)$-currents $S_j$ such that $\partial S_j=\widetilde T_j$ and
\[
\M(S_j)\lesssim \M(\widetilde T_j)\lesssim \M(T_j)+\M(\partial T_j).
\]
Moreover $\M(\partial S_j)=\M(\widetilde T_j)$,
so $S_j$ is a sequence of integral $(k+1)$-currents, supported in $B(0,R+|v|)$, with equibounded mass and mass of the boundary. By the inductive step we deduce that they converge, up to subsequence, to an integral $(k+1)$-current $S$. In particular $\partial S_j\weaksto \partial S$, which is an integral $k$-current.  From \eqref{eq:boundary_explicit} and \eqref{eq:tildeT_j_definition}, sending $j\to\infty$ we find
\begin{equation}\label{eq:three_pieces}
\partial S=(\pi_v)_*(\partial T\times \curr{0,1})+T-(\Tcal_v)_*T,
\end{equation}
where $\Tcal_v$ denotes the translation by the vector $v$.
From the information that $\partial S$ is integer rectifiable we now want to argue that also $T$ is integer rectifiable. The obstacle to this might come from possible cancellations between the three terms at the right-hand side of \eqref{eq:three_pieces}. However:
\begin{enumerate}
    \item the interaction between $T$ and $(\Tcal_v)_*T$ can be eliminated by choosing $|v|$ large enough, as their supports become disjoint;
    \item the current $R_v:=(\pi_v)_*(\partial T\times \curr{0,1})$ is \textit{ruled} in direction $v$, namely its orienting $k$-vector is of the form $v\wedge \tau'$ for some $\tau'$; hence, at worst,  it can  cancel only those pieces of $T$ where $v\in \mathrm{span}(\tau)$. Exploiting the freedom to choose $v$ in any direction, we also rule out cancellations between $R_v$ and $T$.
\end{enumerate}
%We want to use the arbitrariness of $v$ to show that $R$ does not interfere too much with $T$, and thus cannot cancel any non-rectifiable piece.
\begin{figure}
    \centering
    \def\svgwidth{0.5\columnwidth}
    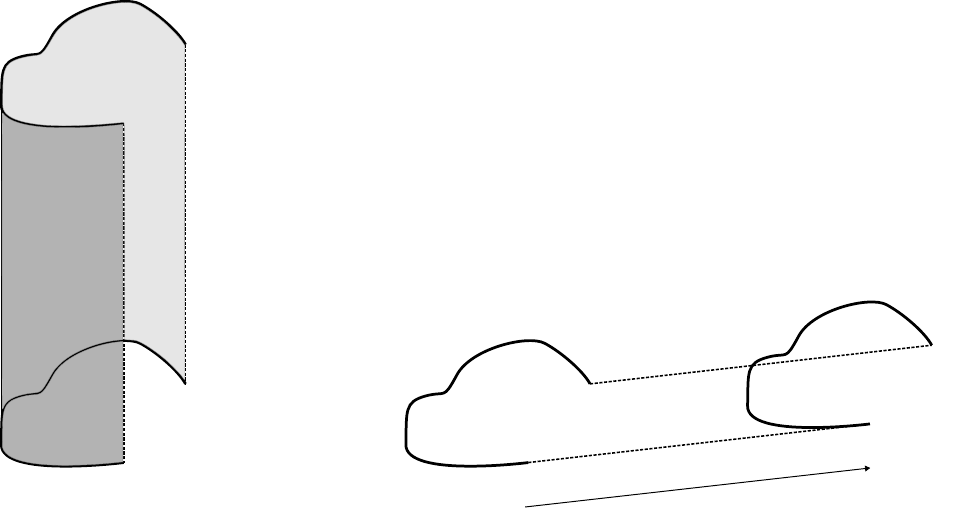

    \caption{\small Cylinder construction. On the left the current $T\times\curr{0,1}$. On the right the projected current $(\pi_v)_*(\partial (T\times \curr{0,1}))$, which is made of three pieces: $T$ and $(\mathcal{T}_v)_*T$ (thick line) and $(\pi_v)_*(\partial T\times \curr{0,1})$ (dashed line).}
    \label{fig:cylinder}
\end{figure}
All in all this implies that the cancellations among these terms are negligible, and that $T$ itself must be integer rectifiable.
We formalize this in the following. Let us first prove that $T$ is concentrated on a rectifiable set. From \eqref{eq:three_pieces} and $R_v\wedge v=0$ it follows that $(T-(\Tcal_v)_*T)\wedge \frac{v}{|v|}$ is concentrated on a rectifiable set, for every choice of $v\ne 0$. By choosing $|v|$ large enough the supports of $T$ and $(\Tcal_v)_*T$ become disjoint, so we also deduce that, for every $e\in \R^n$, $T\wedge e$ is concentrated on a rectifiable set. This is enough to conclude that $T$ itself is concentrated on a rectifiable set, using for instance the inequality
\begin{equation}\label{eq:mass_inequality}
    \|T\|\le \frac{1}{\sqrt{n-k}} \sum_{i=1}^n \|T\wedge e_i\|
\end{equation}
valid for $k\le n-1$. Indeed, writing $T=\tau\mu$, and $\tau(x)=\sum_{\ibf} \tau_\ibf(x) e_\ibf$, for every $i\in\{1,\ldots,n\}$ we have
\[
|\tau(x)\wedge e_i|^2=\sum_{\ibf:\, i\not\in \ibf} |\tau_\ibf(x)|^2
\]
hence
\[
\sum_{i=1}^n |\tau(x)\wedge e_i|^2=\sum_{i=1}^n \sum_{\ibf:\, i\not\in \ibf} |\tau_\ibf(x)|^2=(n-k)\sum_\ibf |\tau_\ibf(x)|^2
\]
since for every $k$-index $\ibf$ there are $n-k$ different indices $i$ not belonging to $\ibf$. Taking the square roots leads to \eqref{eq:mass_inequality}.

From now on, we call $E$ a fixed rectifiable set on which $T$ is concentrated. Observe also that, by a standard blowup argument (see \cite[Lemma~2.2]{White}), the orientation of $T$ must agree with the orientation of the approximate tangent space of $E$ at $\|T\|$-a.e. point.

Now for the integer multiplicity. Let us fix $v\in \R^n$, and let $E_v$ be the set of points in $E$ such that $v\not\in\Tan(E,x)$. Since $R_v$ is supported on a rectifiable set $F_v$ whose tangent space at every point instead contains $v$, it follows that $\H^k(E_v\cap F_v)=0$. Here we are exploiting the fact that two $\H^k$-rectifiable sets $E_v$ and $F_v$ of finite measure have the same approximate tangent space at $\H^k$-a.e. point of $E_v\cap F_v$.
Hence the multiplicity of $T$ at $\H^k$-a.e. point of $E_v$ is the same as the one of $\partial S$, i.e., an integer. As $v$ varies in a countable, dense set of directions, the sets $E_v$ cover $\H^k$-almost all $E$, hence we deduce that the multiplicity of $T$ is an integer at $\H^k$-almost every point of $E$. This shows that $T$ is integer rectifiable.

Finally, since $T$ is also normal, by applying the boundary rectifiability theorem we deduce that $T$ is integral.\hfill\qedsymbol

\appendix

\section{De Giorgi's theorem for \texorpdfstring{$\Z$}{Z}-valued \texorpdfstring{$BV$}{BV} functions}\label{sec:DeGiorgi}
This section is devoted to the proof of (a version of) De Giorgi's theorem for $\Z$-valued $BV$ functions. De Giorgi first showed that for a finite perimeter set $E$, at every approximate continuity point of the polar of $D\1_E$, the set $E$ blows up to a half-space $H$, and the measure $|D\1_E|$ blows up to $\H^{n-1}\restrict\partial H$. The goal of this section is to extend this statement to $\Z$-valued $BV$ functions, at least under some additional assumptions that will be sufficient for our purposes (see Theorem \ref{thm:blowup}). The proof follows very closely the original one as can be found for instance in \cite{AFP}, but we decided to include it as we did not find any reference for the $\Z$-valued case in the literature, and moreover there are some subtleties not present in the original proof, for instance to show uniqueness of the blowup in Proposition \ref{prop:uniqueness}. 

Let $u\in BV(\R^n,\mathbb{Z})$. A point $x_0$ in the support of $|Du|$ is called an approximate continuty point of the polar if
\begin{equation}\label{eq:nu_definition}
\lim_{r\to 0}\frac{Du(B_r(x_0))}{|Du|(B_r(x_0))}=\nu\qquad\text{for some $\nu\in\S^{n-1}$.}
\end{equation}
Given $\nu\in\S^{n-1}$, we denote by $H_\nu$ the half space 
\[
H_\nu:=\{x\in \R^n:\, x\cdot \nu\ge 0\}.
\]
Finally, given $x\in \R^n$ and $r>0$, we denote by $Z_{x,r}:\R^n\to \R^n$ the zooming map
\[
Z_{x,r}(y):=\frac{y-x}{r}
\]
and set $u_{x,r}(y):=u(x+ry)$.
The following is the main result of this section. Observe that in the applications arising from the cylindrical projection, the projected current is supported on one side of the tangent hyperplane, hence why we also require that $u$ is supported on a half-space. It is likely that such condition may be weakened.

\begin{theorem}[Blowup for $\Z$-valued $BV$ functions]\label{thm:blowup}
    Let $u\in BV(\R^n,\mathbb{Z})$. Let $x_0$ satisfy \eqref{eq:nu_definition}, and suppose that, for the same $\nu$, $\supp u\subseteq x_0+H_\nu$. Suppose further that    \begin{equation}\label{eq:density_assumption}
        \Theta^{*,n-1}(|Du|,x_0)<\infty.
    \end{equation}
    Then there exists a non-zero integer $m$ such that:
    \begin{enumerate}[{\rm (i)}]
        \item     
        $u_{x_0,r}\to m\1_{H_\nu}$ in $L^1_{\loc}$ as $r\to 0$.
    
        \item 
    \[
        |Du_{x_0,r}|=\tfrac{1}{r^{n-1}}(Z_{x_0,r})_\# |Du|\weaksto |m|\H^{n-1}\restrict \partial H_\nu\qquad\text{as $r\to 0$}.
    \]
    \item The density $\Theta^{n-1}(|Du|,x_0)$ exists and equals $|m|$.
    \end{enumerate}
\end{theorem}

We recall that the class of $\Z$-valued $BV$ functions in $\R^n$ coincides with the class of integral $n$-currents. Some of the following results can be seen in this light, but we decided to keep the exposition the simplest possible and thus for this section we will only refer to the theory of finite perimeter sets or $BV$ functions in $\R^n$. For instance, the lower density in Lemma \ref{lemma:positive_lower_density} could be obtained directly for the current $T$ (at $\|T\|$-a.e. point) by slicing and the isoperimetric inequality (see \cite[2.1]{White}).

\begin{proposition}[Global isoperimetric inequality]\label{prop:global_isoperimetric}
    There exists a constant $C_I$ such that for every $u\in BV(\R^n;\Z)$
    \[
    \|u\|_{L^1(\R^n)} \le C_I |Du|(\R^n)^\frac{n}{n-1}.
    \]
\end{proposition}

\begin{proof}
    We apply Poincar\'e's inequality (see Theorem~3.44 and Remark~3.45 in \cite{AFP})
    \[
    \int_Q |u-\bar u_Q|\le C_P \ell(Q) |Du|(Q)
    \]
    on cubes of side length $r_0:=(2\|u\|_{L^1(\R^n)})^\frac{1}{n}$. For every such cube $Q$ we have
    \[
    |\bar u_Q|\le \frac{1}{|Q|}\int_Q |u|\le \frac12.
    \]
    It follows that 
    \begin{align*}
        \fint_Q |u-\bar u_Q|& =\frac{1}{|Q|}\int_Q |u-\bar u_Q|\ge \frac12\frac{1}{|Q|}\int |u|
        %=\frac{1}{|Q|}\left(\int_{Q\cap\{u=0\}}|\bar u_Q|+\int_{Q\cap\{u\ne 0\}} |u-\bar u_Q|\right)\\
        %&\ge \frac{1}{|Q|}\int_{Q\cap\{u\ne0\}} \frac12 |u|\\
        %&=\frac12 \frac{1}{|Q|}\int_Q |u|,
    \end{align*}
    where in the last step we have used the elementary property
    \[
    \begin{cases}
        m\in\R,\,|m|\le\frac12\\
        n\in \Z
    \end{cases}\implies |m-n|\ge \frac12|n|.
    \]
    Using again Poincar\'e's inequality we deduce that for every cube $Q$ of side length $r_0$
    \[
        \int_Q |u|\le 2\int_Q|u-\bar u_Q|\le 2C_P r_0 |Du|(Q).
    \]
    Summing this inequality over a partition of $\R^n$ in cubes of side length $r_0$ gives
    \[
    \|u\|_{L^1(\R^n)}\le 2 C_P r_0 |Du|(\R^n),
    \]
    which recalling our choice of $r_0$ yields the thesis with $C_I=(2C_P)^\frac{n}{n-1}$.
\end{proof}

\subsection{Density bounds}
The next lemmas show that, at points of interest for us, the lower density $\Theta_*^{n}(|u|,x_0)$ is positive and the upper density $\Theta^{*,n}(|u|,x_0)$ is finite. They also apply to $BV$ functions that attain real values, and not only integer values.

\begin{lemma}[Upper density]\label{lemma:upper_densities_I}
    Assume that $u\in BV(\R^n)$ and that $x_0$ is a point such that $\supp(u)\subseteq x_0+H_\nu$ for some $\nu$. Then for every $r>0$
    \[
    \fint_{B_r(x_0)}|u(x)|dx\le C \frac{|Du|(B_r(x_0))}{r^{n-1}}
    \]
    for some dimensional constant $C$.
\end{lemma}

\begin{proof}
This follows directly from Poincar\'{e}'s inequality in the space of functions that vanish in a half-ball. To show it, set for simplicity $B_r=B_r(x_0)$, $N:=B_r\cap H_\nu$, $\bar u:=\fint_{B_r} u$. First we have that
\begin{equation}\label{eq:Poincare_N}
\frac12 | \bar u|=\frac{1}{2}\fint_{N}|u-\bar u|\le \fint_{B_r} |u-\bar u|\le C\frac{|Du|(B_r)}{r^{n-1}}
\end{equation}
by the classical Poincar\'{e}'s inequality. Then
\[
\fint_{B_r} |u|\le \fint_{B_r} |u-\bar u|+\fint_{B_r}|\bar u| =\fint_{B_r} |u-\bar u|+|\bar u|
\]
and the conclusion follows applying again Poincar\'e's inequality and \eqref{eq:Poincare_N}.
\end{proof}

\begin{lemma}[Positive lower density]\label{lemma:positive_lower_density}
    Let $u\in BV(\R^n)$. Let $x_0\in \supp Du$ satisfy \eqref{eq:nu_definition}. Then
    \begin{equation}\label{eq:lower_density_u}
    \liminf_{r\to 0}\fint_{B_r(x_0)} |u(x)|dx\ge c_1
    \end{equation}
    for some dimensional constant $c_1>0$.
    %In particular, $\Theta_*^{n-1}(|Du|,x_0)>0$.
\end{lemma}

\begin{proof}
    We follow closely the proof of \cite[Lemma~3.58]{AFP}. Set
    \[
    m(r):=\int_{B_r(x_0)}|u(x)|dx,
    \]
    and $u_r:=u\1_{B_r(x_0)}$. First, by \cite[Equation~(3.50)]{AFP}, for every $r>0$
    \[
    |Du_r|(\R^n)\le |Du|(\overline{B}_r(x_0))+m'_+(r),
    \]
    where $m'_+$ is the lower right derivative.
    Using the locality of the total variation we deduce
    \begin{align*}
        |Du|(B_r(x_0))+|Du_r|(\partial B_r(x_0))& =|Du_r|(B_r(x_0))+|Du_r|(\partial B_r(x_0))\\
        & =|Du_r|(\overline B_r(x_0))\\
        &= |Du_r|(\R^n)\\
        & \le |Du|(\overline{B}_r(x_0)) +m'_+(r).
    \end{align*}    
    In particular, for every $r>0$ such that $|Du|(\partial B_r(x_0))=0$, subtracting $|Du|(B_r(x_0))$ to both sides we get
    \[
    |Du_r|(\partial B_r(x_0))\le m'_+(r).
    \]
    Moreover, by \eqref{eq:nu_definition}, for every sufficiently small $r>0$ it holds $|Du|(B_r(x_0))\le 2 |Du(B_r(x_0)|$.
    Additionally, 
    \[
    0=Du_r(\overline{B_r(x_0)})=Du(B_r(x_0))+Du_r(\partial B_r(x_0)).
    \]
    Putting everything together we obtain
    \begin{align*}
    |Du_r|(\R^n)& =|Du_r|(B_r(x_0))+|Du_r|(\partial B_r(x_0))\\
    & = |Du|(B_r(x_0))+|Du_r|(\partial B_r(x_0))\\
    &\le 2|Du(B_r(x_0))|+|Du_r|(\partial B_r(x_0))\\
    &\le 2|Du_r(\partial B_r(x_0))|+|Du_r|(\partial B_r(x_0)) \\
    &\le 3m'_+(r).
    \end{align*}
    Then applying the global isoperimetric inequality given by Proposition~\ref{prop:global_isoperimetric} to $u_r$ we find that for $\Lcal^1$-a.e. $r\in (0,r_0)$
    \begin{align*}
        (m^{1/n})'(r)=\frac{1}{n}m^{(1-n)/n}(r)m'(r)\ge \frac{1}{3n} m^{(1-n)/n}(r)|Du_r|(\R^n)\ge \frac{1}{3nC}.
    \end{align*}
    Integrating in $r$ gives \eqref{eq:lower_density_u}. %The last assertion about $\Theta_*^{n-1}(|Du|,x_0)$ follows from Lemma \ref{lemma:upper_densities_I}.
\end{proof}

\subsection{Uniqueness of blowups}

We now come to a result on the uniqueness of the blowup for $\Z$-valued $BV$ functions at approximate continuity points of $Du$, under a finite density bound. The proof of this result might have some interest in itself, and may be reminiscent of some arguments used in the proof of Preiss' theorem (see \cite[Section~4]{delellis}). Observe that this step is unnecessary for sets of finite perimeter, as the only non-decreasing characteristic function on $\R$ that has a jump in 0 is the characteristic of the positive half-line, while for $\mathbb{Z}$-valued functions there could be multiple jumps. %It is somehow reminiscent of the proof of ....\red{Result interesting by itslef??}

We call \textit{blowup sequence} for $u$ at $x_0$ any sequence of the form $u_{x_0,r_j}$, for some $r_j\to 0$. We call \textit{blowup} of $u$ at $x_0$ any $L^1_\loc$-limit of a blowup sequence, and we denote the set of all such blowups by $\Tan(u,x_0)$. We will repeatedly use the following elementary facts:
\begin{enumerate}
    \item By Lemma \ref{lemma:upper_densities_I} and the compactness of $BV$ functions, if $\Theta^{*,n-1}(|Du|,x_0)<\infty$ and $\supp(u)\subseteq x_0+H_\nu$ for some $\nu$, then every blowup sequence $u_{x_0,r_j}$ has a subsequence converging in $L^1_\loc$.
    \item If $v\in \Tan(u,x_0)$ then for every $r>0$ the dilation $v_{0,r}$ belongs to $\Tan(u,x_0)$.
\end{enumerate}

\begin{proposition}[Uniqueness of blowups]\label{prop:uniqueness}
    Let $u\in BV(\R^n,\Z)$ such that $\supp u\subseteq x_0+H_\nu$, and let $x_0$ satisfy \eqref{eq:nu_definition}. Suppose that $\Theta^{*,n-1}(|Du|,x_0)<\infty$. Then the blowup at $x_0$ is unique.
\end{proposition}

The strategy is to exploit the fact that every blowup at $x_0$ needs to be monodirectional and monotone, and to show that if two distinct blowups existed, then one could construct a third blowup violating monotonicity. We realize this by comparing averages of $u$ along a chain of consecutively tangent balls. These averages need to oscillate at different scales because of the existence of two distinct blowups, and this creates an issue because: either at many scales the oscillation from a ball to the next one goes in the ``wrong'' direction from the point of view of the monotonicity; or at many scales the averages on two consecutive tangent balls do not change much, hence at some point are bound to remain close to a non-integer value, in which case the integer-valuedness of $u$ still forces some oscillation in the wrong direction. This allows for the creation of a non-monotone blowup and thus a contradiction.

\begin{proof}
    From the approximate continuity assumption we deduce that every blowup $v$ is monodirectional and monotone, namely there exists a nondecreasing $h_v:\R\to\R$ (depending a priori on $v$) such that
    \[
    v(x)=h_v(x\cdot\nu).
    \]
    For a proof of this, see e.g. \cite[Theorem~3.59]{AFP}.     
    Since $u$ is $\Z$-valued, and since any $L^1$-convergent sequence has a subsequence that converges almost everywhere, we deduce that also $v$ and $h_v$ are $\Z$-valued. From the upper density assumption we also deduce that 
    \[
    M:=\max\{\|h_v\|_{L^\infty}:\, v\in \Tan(u,x_0)\}
    \]
    is finite. We can assume that $M>0$, otherwise every blowup is trivial and the conclusion holds.
    
    We want to exploit the monotonicity property above, showing that if the blowup were not unique then we could find an additional blowup which is not monotone.
    Suppose thus by contradiction that there exist two distinct blowups $v$ and $w$ at $x_0$. We can assume that $\|v\|_{L^\infty}=M$. In particular, up to a dilation of the domain (and possibly changing the sign of $u$, which does not change the argument), this implies that there exists a ball $B_\eps(\nu)$, $\eps>0$, such that
    \[
    v=M\text{ a.e. on $B_\eps(\nu)$}\qquad \text{and}\qquad w=m\text{ a.e. on $B_\eps(\nu)$},
    \]
    for some $m\in\mathbb{Z}$, $|m|<M$. This entails the existence of two infinitesimal sequences $r_j$ and $s_j$ such that
    \begin{equation}\label{eq:mM}
    \int_{B( \nu r_j, \eps r_j) } |u-M|\to 0\qquad\text{and}\qquad \int_{B( \nu s_j, \eps s_j) } |u-m|\to 0
    \end{equation}
    Let us now define a family of consecutively tangent balls (see Figure \ref{fig:balls})
    \[
    B_i:=\tau^i B(\nu,\tilde\eps),\qquad i\in \mathbb{N}, \,\tau:=\frac{1-\tilde\eps}{1+\tilde\eps},
    \]
    for some $\tilde\eps>0$, and let us define correspondingly
    $u_i:=\fint_{B_i} u$.
    If $\tilde\eps$ is chosen small enough, then every ball $B(\nu r_j,\eps r_j)$ appearing in \eqref{eq:mM} contains entirely one of the balls $B_i$ defined above, and the same for $B(\nu s_j,\eps s_j)$. Hence the same estimates \eqref{eq:mM} hold also for subsequences of $B_i$, i.e., we can find two interlaced sequences of natural numbers $(i_k), (j_k)$ such that
    $j_k<i_k<j_{k+1}<i_{k+1}$
    and for which 
    \begin{equation}\label{eq:mM_B_i}
        \int_{B_{i_k} } |u-M|\to 0\qquad\text{and}\qquad \int_{B_{j_k}} |u-m|\to 0.
    \end{equation}
    We claim that for $i$ large enough 
    \begin{equation}\label{eq:1/10}
    u_{i+1}\le u_i +\frac{1}{10}.
    \end{equation}
    Assume on the contrary that $u_{i+1}\ge u_i+\frac{1}{10}$ for infinitely many $i$'s, then we can take the blowup sequence corresponding to $B(x_0,r_i)$, $r_i:=\tau^i |\nu|$, pass to a strong $L^1$-limit $g$ up to subsequence, and obtain that for $g$
    \[
    \fint_{B_1} g\ge \fint_{B_0} g+\frac{1}{10},
    \]
    which contradicts the monotonicity of $g$ in direction $\nu$. Therefore \eqref{eq:1/10} is proven.

\begin{figure}
        \centering
    \def\svgwidth{0.4\columnwidth}
    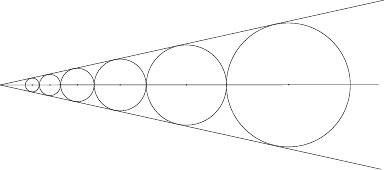

    \caption{\small The family of consecutively tangent balls $B_i$.}\label{fig:balls}
    \end{figure}
    
    Next observe that the oscillation between the two blowups forces the averages to repeatedly cross from values substantially below $M$ to values near $M$, while \eqref{eq:1/10} prevents this transition from occurring too abruptly. Hence there must be a lot of scales for which the average on two consecutive balls is close to $M-\tfrac12$. More precisely, we claim that we can find infinitely many $i$'s for which
    \begin{equation}\label{eq:controlled_average}
    M-\frac14 \ge u_{i+1}\ge u_i \ge M-\frac34.
    \end{equation}
    Observe first that from \eqref{eq:mM_B_i}, if $k$ is large enough we get 
    \[
    u_{i_k}\ge M-\frac14\qquad u_{j_k}\le m+\frac14\le M-\frac34.
    \]
    Up to relabeling the sequences, we will assume that the previous estimate holds for every $k\ge 1$. 
    We now look at the last scale before the averages cross the midpoint level $M-\frac12$; the near-monotonicity estimate \eqref{eq:1/10} then forces two consecutive averages to remain trapped in a narrow interval away from the integers.
    To this end, consider 
    \[
    \ell_k:=\max\left\{\ell\in \mathbb{N}, \, i<i_k:\, u_\ell\le M-\frac12\right\}\in \{j_k,\ldots,i_k-1\}.
    \]
    Then using \eqref{eq:1/10} and the definition of $\ell_k$
    \[
    u_{\ell_k}\le M-\frac12\le u_{\ell_k+1}\le u_{\ell_k}+\frac{1}{10}
    \]
    hence
    \[
    u_{\ell_k}\ge M-\frac12-\frac{1}{10}\ge M-\frac34.
    \]
    On the other hand,
    \[
    u_{\ell_k+1}\le u_{\ell_k}+\frac{1}{10}\le M-\frac12 +\frac{1}{10}\le M-\frac14.
    \]
    Hence the index $\ell_k$ satisfies \eqref{eq:controlled_average}, and the claim is proven.

    Now let us consider a blowup $g$ associated to the sequence $r_{\ell_k}:=\tau^{\ell_k}$. From \eqref{eq:controlled_average} we obtain that 
    \[
    M-\frac14 \ge g_{B_1}\ge g_{B_0}\ge M-\frac34.
    \]
    In this final step we exploit the integer-valuedness of $g$: on the one hand
    \begin{align*}
        M-\frac34 \le g_{B_1} &\le M\frac{|\{g=M\}\cap B_1|}{|B_1|}+(M-1)\frac{|\{g<M\}\cap B_1|}{|B_1|}\\
        &= M\frac{|\{g=M\}\cap B_1|}{|B_1|} +(M-1)\left(1-\frac{|\{g=M\}\cap B_1|}{|B_1|}\right)%\\
        %&=M-1+\frac{|\{v=M\}\cap B_1|}{|B_1|}
    \end{align*}
    hence rearranging terms
    \[
    \frac{|\{g=M\}\cap B_1|}{|B_1|}\ge \frac14.
    \]
    On the other hand
    \begin{align*}
        M-\frac14 \ge g_{B_0}&\ge M \frac{|\{g=M\}\cap B_0|}{|B_0|}-M \frac{|\{g<M\}\cap B_0|}{|B_0|}\\
        & = M \frac{|\{g=M\}\cap B_0|}{|B_0|}-M\left(1-\frac{|\{g=M\}\cap B_0|}{|B_0|}\right)%\\
        %&= 2M \frac{|\{v=M\}\cap B_0|}{|B_0|}-M
    \end{align*}
    hence rearranging terms
    \[
    \frac{|\{g=M\}\cap B_0|}{|B_0|}\le 1-\frac{1}{8M}.
    \]
    This shows that $g$ attains values lower than $M$ on a significant portion of $B_0$, while it attains value $M$ on a significant portion of $B_1$. This again contradicts the monotonicity of $g$. In conclusion, we found a contradiction to the assumption that there exist two distinct blowups at $x_0$, thus uniqueness is proved.
\end{proof}

We can finally put together all the results above to give a proof of the main theorem of this section.

\subsection{Proof of Theorem \ref{thm:blowup}}
    By Proposition \ref{prop:uniqueness} we deduce that the blowup of $u$ at $x_0$ is unique, and thus of the form $v(x)=m\1_{H_\nu}$, for some $m\in\mathbb{Z}$. Indeed, by uniqueness it must be $0$-homogeneous (as every dilation $v_{0,r}$ must coincide with $v$) and by monodirectionality and monotonicity the only option is $v(x)=m\1_{H_\nu}$. By Lemma \ref{lemma:positive_lower_density} we deduce also that $m\neq 0$, so (i) follows. By uniqueness of the blowup we know that $Du_{x_0,r}\weaksto Dv$ as $r\to 0$, and hence by the approximate continuity property \eqref{eq:nu_definition} and \cite[Theorem~2.44]{AFP} it also holds
    \[
    |Du_{x_0,r}|\weaksto |Dv|=|m|\H^{n-1}\restrict \partial H_\nu\qquad\text{as $r\to 0$}.
    \]
    This shows (ii). Finally, (iii) follows from (ii) and standard continuity results under weak* convergence of measures, observing that the limit measure does not charge the boundary of balls.\hfill \qedsymbol

\bibliographystyle{alpha}
\bibliography{biblio.bib}

\end{document}

%% file: cylindrical_projection.pdf_tex
%% Creator: Inkscape 1.2.2 (b0a8486541, 2022-12-01), www.inkscape.org
%% Accompanies image file 'cylindrical_projection.pdf' (pdf, eps, ps)
\begingroup%
  \makeatletter%
  \providecommand\color[2][]{%
    \errmessage{(Inkscape) Color is used for the text in Inkscape, but the package 'color.sty' is not loaded}%
    \renewcommand\color[2][]{}%
  }%
  \providecommand\transparent[1]{%
    \errmessage{(Inkscape) Transparency is used (non-zero) for the text in Inkscape, but the package 'transparent.sty' is not loaded}%
    \renewcommand\transparent[1]{}%
  }%
  \providecommand\rotatebox[2]{#2}%
  \newcommand*\fsize{\dimexpr\f@size pt\relax}%
  \newcommand*\lineheight[1]{\fontsize{\fsize}{#1\fsize}\selectfont}%
  \ifx\svgwidth\undefined%
    \setlength{\unitlength}{439.71545602bp}%
    \ifx\svgscale\undefined%
      \relax%
    \else%
      \setlength{\unitlength}{\unitlength * \real{\svgscale}}%
    \fi%
  \else%
    \setlength{\unitlength}{\svgwidth}%
  \fi%
  \global\let\svgwidth\undefined%
  \global\let\svgscale\undefined%
  \makeatother%
  \begin{picture}(1,0.73691861)%
    \lineheight{1}%
    \setlength\tabcolsep{0pt}%
    \put(0,0){\includegraphics[width=\unitlength,page=1]{cylindrical_projection.pdf}}%
    \put(0.2030058,0.36420224){\makebox(0,0)[lt]{\lineheight{1.25}\smash{\begin{tabular}[t]{l}$x_0$\end{tabular}}}}%
    \put(0.19816513,0.46011728){\makebox(0,0)[lt]{\lineheight{1.25}\smash{\begin{tabular}[t]{l}$\tau_0$\end{tabular}}}}%
    \put(0,0){\includegraphics[width=\unitlength,page=2]{cylindrical_projection.pdf}}%
    \put(0.78417405,0.35859865){\makebox(0,0)[lt]{\lineheight{1.25}\smash{\begin{tabular}[t]{l}$0$\end{tabular}}}}%
    \put(0.77434921,0.45451364){\makebox(0,0)[lt]{\lineheight{1.25}\smash{\begin{tabular}[t]{l}$e_1$\end{tabular}}}}%
    \put(0,0){\includegraphics[width=\unitlength,page=3]{cylindrical_projection.pdf}}%
  \end{picture}%
\endgroup%

%% file: Slanted_projection.pdf_tex
%% Creator: Inkscape 1.2.2 (b0a8486541, 2022-12-01), www.inkscape.org
%% Accompanies image file 'Slanted_projection.pdf' (pdf, eps, ps)
\begingroup%
  \makeatletter%
  \providecommand\color[2][]{%
    \errmessage{(Inkscape) Color is used for the text in Inkscape, but the package 'color.sty' is not loaded}%
    \renewcommand\color[2][]{}%
  }%
  \providecommand\transparent[1]{%
    \errmessage{(Inkscape) Transparency is used (non-zero) for the text in Inkscape, but the package 'transparent.sty' is not loaded}%
    \renewcommand\transparent[1]{}%
  }%
  \providecommand\rotatebox[2]{#2}%
  \newcommand*\fsize{\dimexpr\f@size pt\relax}%
  \newcommand*\lineheight[1]{\fontsize{\fsize}{#1\fsize}\selectfont}%
  \ifx\svgwidth\undefined%
    \setlength{\unitlength}{464.61869367bp}%
    \ifx\svgscale\undefined%
      \relax%
    \else%
      \setlength{\unitlength}{\unitlength * \real{\svgscale}}%
    \fi%
  \else%
    \setlength{\unitlength}{\svgwidth}%
  \fi%
  \global\let\svgwidth\undefined%
  \global\let\svgscale\undefined%
  \makeatother%
  \begin{picture}(1,0.54936478)%
    \lineheight{1}%
    \setlength\tabcolsep{0pt}%
    \put(0,0){\includegraphics[width=\unitlength,page=1]{Slanted_projection.pdf}}%
    \put(0.90403556,0.01778659){\makebox(0,0)[lt]{\lineheight{1.25}\smash{\begin{tabular}[t]{l}$v$\end{tabular}}}}%
    \put(0.47490703,0.22676339){\makebox(0,0)[lt]{\lineheight{1.25}\smash{\begin{tabular}[t]{l}$T$\end{tabular}}}}%
    \put(0.83257178,0.27103983){\makebox(0,0)[lt]{\lineheight{1.25}\smash{\begin{tabular}[t]{l}$(\mathcal{T}_v)_*T$\end{tabular}}}}%
    \put(0.06984354,0.005441){\makebox(0,0)[lt]{\lineheight{1.25}\smash{\begin{tabular}[t]{l}$T\times\curr{0,1}$\end{tabular}}}}%
  \end{picture}%
\endgroup%

%% file: balls.pdf_tex
%% Creator: Inkscape 1.2.2 (b0a8486541, 2022-12-01), www.inkscape.org
%% Accompanies image file 'balls.pdf' (pdf, eps, ps)
\begingroup%
  \makeatletter%
  \providecommand\color[2][]{%
    \errmessage{(Inkscape) Color is used for the text in Inkscape, but the package 'color.sty' is not loaded}%
    \renewcommand\color[2][]{}%
  }%
  \providecommand\transparent[1]{%
    \errmessage{(Inkscape) Transparency is used (non-zero) for the text in Inkscape, but the package 'transparent.sty' is not loaded}%
    \renewcommand\transparent[1]{}%
  }%
  \providecommand\rotatebox[2]{#2}%
  \newcommand*\fsize{\dimexpr\f@size pt\relax}%
  \newcommand*\lineheight[1]{\fontsize{\fsize}{#1\fsize}\selectfont}%
  \ifx\svgwidth\undefined%
    \setlength{\unitlength}{184.22437292bp}%
    \ifx\svgscale\undefined%
      \relax%
    \else%
      \setlength{\unitlength}{\unitlength * \real{\svgscale}}%
    \fi%
  \else%
    \setlength{\unitlength}{\svgwidth}%
  \fi%
  \global\let\svgwidth\undefined%
  \global\let\svgscale\undefined%
  \makeatother%
  \begin{picture}(1,0.44162568)%
    \lineheight{1}%
    \setlength\tabcolsep{0pt}%
    \put(0,0){\includegraphics[width=\unitlength,page=1]{balls.pdf}}%
  \end{picture}%
\endgroup%